\documentclass[12pt]{amsart}

\usepackage[latin1]{inputenc}
\usepackage{amsmath}
\usepackage{amsfonts}
\usepackage{amssymb}
\usepackage{graphics}
\usepackage{enumerate}
\usepackage{amssymb,amsmath,amsthm,amscd,epsf,latexsym,verbatim,graphicx,amsfonts}
\input epsf.tex

\usepackage{amscd}
\usepackage{amsmath}
\usepackage{amssymb}
\usepackage{amsthm}
\usepackage{epsf}
\usepackage{latexsym,bm}
\usepackage{verbatim}
\input epsf.tex

\newtheorem{theorem}{Theorem}[section]

\numberwithin{equation}{section}

\newcommand{\bbR}{\mathbb{R}}

\newcommand{\bbD}{\mathbb{D}}
\newcommand{\bbN}{\mathbb{N}}

\newcommand{\eitheta}{e^{i\theta}}

\newcommand{\eitau}{e^{i\tau}}

\newcommand{\eilambda}{e^{i\lambda}}

\newcommand{\Phib}{\Phi_{n+1}^{(\beta_n)}}

\newcommand{\nri}{n\rightarrow\infty}
\newcommand{\jri}{j\rightarrow\infty}
\newcommand{\bsa}{\boldsymbol{\alpha}}

\begin{document}

\title[ ] {Zero Spacing of Paraorthogonal Polynomials on the Unit Circle}

\bibliographystyle{plain}

\thanks{  }


\maketitle




\begin{center}
\textbf{Brian Simanek}
\end{center}


\begin{abstract}
We prove some new results about the spacing between neighboring zeros of paraorthogonal polynomials on the unit circle.  Our methods also provide new proofs of some existing results.  The main tool we will use is a formula for the phase of the appropriate Blaschke product at points on the unit circle.
\end{abstract}


\vspace{4mm}


\vspace{2mm}


\vspace{2mm}

\normalsize

\section{Introduction}\label{intro}

Given an infinitely supported probability measure $\mu$ on the unit circle $\partial\bbD$, it is well known how to obtain the corresponding monic orthogonal polynomials $\{\Phi_n(z)\}_{n=0}^{\infty}$.  One key feature of orthogonal polynomials on the unit circle (OPUC) is the \textit{Szeg\H{o} recursion}, which states that for each $n\in\{0,1,2,\ldots\}$ there is a complex number $\alpha_n$ such that $|\alpha_n|<1$ and
\begin{equation}\label{opucdef}
\Phi_{n+1}(z)=z\Phi_n(z)-\overline{\alpha}_n\Phi_n^*(z),
\end{equation}
where $\Phi_n^*(z)=z^n\overline{\Phi_n(1/\overline{z})}$.  The sequence $\{\alpha_n\}_{n=0}^{\infty}$ is often called the sequence of \textit{Verblunsky coefficients} and Verblunsky's Theorem asserts that there is a bijection between infinite sequences in $\bbD$ and infinitely supported probability measures on the unit circle (see \cite[Section 1.7]{OPUC1}).

We will be interested in a related sequence of polynomials called the \textit{paraorthogonal polynomials on the unit circle} (POPUC), which were introduced in \cite{NJT} (see also \cite[Section 2.2]{OPUC1}).  Given the sequence $\{\Phi_n(z)\}_{n=0}^{\infty}$, one defines the degree $n+1$ paraorthogonal polynomial $\Phi_{n+1}^{(\beta_n)}(z)$ by
\begin{equation}\label{popucdef}
\Phi_{n+1}^{(\beta_n)}(z)=z\Phi_n(z)-\overline{\beta}_n\Phi_n^*(z),
\end{equation}
where $|\beta_n|=1$.  Thus, to every sequence $\{\alpha_n\}_{n=0}^{\infty}$ in $\bbD$ and every sequence $\{\beta_n\}_{n=0}^{\infty}$ in $\partial\bbD$, one can define the sequence of paraorthogonal polynomials $\{\Phi_{n+1}^{(\beta_n)}(z)\}_{n=0}^{\infty}$ using \eqref{opucdef} and \eqref{popucdef}.

Paraorthogonal polynomials have been well studied recently because of their relevance to a variety of applications.  These include their role in numerical quadrature formulas (see \cite{Golquad,NJT}), their description of solutions to certain electrostatics problems on the circle (see \cite{SimElec}), and their relationship to Poncelet polygons (see \cite{AMFSS}).  We will be interested in properties of the zeros of these polynomials and there exists a substantial literature on that subject.  Many results have been proven about interlacing properties (see \cite{CMR,RankOne,Lilian}), zero spacing (see \cite{Golquad,FineI,FineIV,NoBax}), and the relationship between the zeros and the underlying measure of orthogonality (see \cite{MFSRV,RankOne}).
Our primary interest will be in the spacing between zeros, though we will also have something to say about the bulk distribution of zeros.


The formula \eqref{popucdef} tells us that
\[
\Phib=0\qquad\iff\qquad\frac{\eitheta\Phi_n(\eitheta)}{\Phi_n^*(\eitheta)}=\overline{\beta}_n.
\]
Define $b_n(z):=\frac{\Phi_n(z)}{\Phi_n^*(z)}$, which is easily seen to be a Blaschke product.  From this observation, we see that all zeros of $\Phi_{n+1}^{(\beta_n)}$ lie on the unit circle and are simple.  Let us define the \textit{Pr\"{u}fer phase} $\eta_n$ of this Blaschke product by
\begin{equation}\label{pruferdef}
e^{i\eta_n(\theta)}=\eitheta b_n(\eitheta).
\end{equation}
Similar objects have also been studied in \cite{KStoi,FineIV,Milvoy,NoBax}.  Though our definition only defines $\eta_n$ up to an integer multiple of $2\pi$, we will be looking at changes in $\eta_n(\theta)$ as $\theta$ changes, so our choice of normalization will be irrelevant.  Formula 10.8 in \cite{FineIV} implies $\eta_n$ is strictly increasing and increases by $2(n+1)\pi$ as $\theta$ runs through any real interval of length $2\pi$.  We will also note here that if $\beta_n\in\partial\bbD$, then any interval $[a,b]\subseteq\bbR$ for which $\eta_n(b)-\eta_n(a)\geq2\pi$ contains a value $t$ so that $\Phi_{n+1}^{(\beta_n)}(e^{it})=0$.


The key to our analysis will be a recursion relation satisfied by $\{\eta_n\}_{n=0}^{\infty}$, which we now derive.  Let us set $\eta_0(\theta)=\theta$ for all $\theta\in\bbR$.  From our definition of $\eta_n$ and the relation
\[
b_{n+1}(z)=\frac{zb_n(z)-\bar{\alpha}_n}{1-\alpha_nzb_n(z)}
\]
(see \cite[Eq. 9.2.16]{OPUC2}), we find
\[
e^{i\eta_{n+1}(\theta)}=e^{i(\eta_n(\theta)+\theta)}\left(\frac{1-\bar{\alpha}_ne^{-i\eta_n(\theta)}}{1-\alpha_ne^{i\eta_n(\theta)}}\right)=e^{i\eta_n(\theta)+i\theta}\left(\frac{1-\bar{\alpha}_ne^{-i\eta_n(\theta)}}{|1-\alpha_ne^{i\eta_n(\theta)}|}\right)^2
\]
from which it follows that
\begin{eqnarray}\label{key}
\eta_{n+1}(\theta)=\eta_n(\theta)+\theta-2\arg[1-\alpha_ne^{i\eta_n(\theta)}]
\end{eqnarray}
where we let $\arg$ take values in $(-\pi,\pi]$ (note the similarity between this recurrence and that of \cite[Proposition 2.2]{KStoi}).  Since each $\alpha_n\in\bbD$, we see that in this recursion, the $\arg$ function will only take values in $(-\pi/2,\pi/2)$.  

Iterating the recursion \eqref{key} allows us to write
\begin{eqnarray}\label{sumrecur}
\eta_{n+1}(\theta)&=&(n+2)\theta-2\sum_{j=0}^{n}\arg[1-\alpha_jb_j(\eitheta)].
\end{eqnarray}
Observe that in this formula it is true that $\eta_{n+1}(\theta+2\pi)=\eta_n(\theta)\mod 2\pi$.  Thus, if we define $\eta_0(\theta)=\theta$ for all $\theta\in\bbR$ and then iteratively define $\eta_n(\theta)$ for all $n\in\bbN$ and $\theta\in\bbR$ by \eqref{sumrecur}, then \eqref{pruferdef} holds for all $n\in\bbN\cup\{0\}$ and $\theta\in\bbR$.

Our main goal is to prove several theorems regarding the zeros of POPUC.  The first two are improvements of a 2002 result of Golinskii \cite{Golquad}, which relates decay properties of the Verblunsky coefficients to the separation of zeros of the corresponding POPUC.  Specifically, he proved that if $\{\alpha_n\}_{n=0}^{\infty}\in\ell^2$, then the largest distance between any two zeros of $\Phi_{n+1}^{(\beta_n)}(z)$ is $O(n^{-1/2})$ as $\nri$.  We strengthen this estimate to $o(n^{-1/2})$ as $\nri$ and present an analogous result that applies when $\{\alpha_n\}_{n=0}^{\infty}\in\ell^p$ for some $p\in(1,\infty)$.  We also show that only a weak-type $\ell^2$ condition is required to obtain the $O(n^{-1/2})$ estimate.  We will precisely state and prove these results in Section \ref{baxter}.

Our final result concerns the bulk distribution of zeros of POPUC as $\nri$.  We show that - under an appropriate condition on the Verblunsky coefficients - arcs of a certain length must contain many zeros of $\Phi_{n+1}^{(\beta_n)}(z)$ when $n$ is very large.  This generalizes a result that is sometimes called the Mhaskar-Saff Theorem for POPUC.  We will precisely state and prove this result in Section \ref{mspopuc}.

\section{Zero Spacing}\label{baxter}

In this section, we discuss the relationship between decay properties of the sequence of Verblunsky coefficients and the spacing of consecutive zeros of the paraorthogonal polynomials.  This theorem is of particular interest given the results of \cite{KStoi}, which shows that the rate of decay of the Verblunsky coefficients has a profound influence on zero spacings.

For a sequence of Verblunsky coefficients $\bsa=\{\alpha_j\}_{j=0}^{\infty}$, we denote by $\|\bsa\|_{\ell^p}$ the $\ell^p$ norm of this sequence.  As mentioned above, \cite[Theorem 3]{Golquad} states that if $\|\bsa\|_{\ell^2}<\infty$, then the distance between neighboring zeros of $\Phi_n^{(\beta)}(z)$ is bounded by $C/\sqrt{n}$ for some explicit constant $C$ (depending on $\bsa$, but not on $\beta$ or $n$).
Our first theorem is a strengthening of that result.

\begin{theorem}\label{zerospace}
Suppose $p\in(1,\infty)$ is fixed and $\|\bsa\|_{\ell^p}<\infty$.  For each $n\in\bbN\cup\{0\}$, choose $\beta_n\in\partial\bbD$ and define the paraorthogonal polynomials $\{\Phi_{n+1}^{(\beta_n)}(z)\}_{n=0}^{\infty}$ as in Section \ref{intro}.  Let $\{e^{i\theta_j^{(n+1)}}\}_{j=1}^{n+1}$ be the zeros of $\Phi_{n+1}^{(\beta_{n})}(z)$ arranged so that $0\leq\theta_1^{(n+1)}<\theta_2^{(n+1)}<\cdots<\theta_{n+1}^{(n+1)}$.  Then $($with $\theta_{n+2}^{(n+1)}=\theta_{1}^{(n+1)}+2\pi)$
\[
\lim_{\nri}\left[n^{1/p}\cdot\sup_{1\leq j\leq n+1}\left|\theta_{j+1}^{(n+1)}-\theta_{j}^{(n+1)}\right|\right]=0.
\]
\end{theorem}

\begin{proof}
It suffices to show that for any $C>0$ it is true that
\[
\sup_{1\leq j\leq n+1}\left|\theta_{j+1}^{(n+1)}-\theta_{j}^{(n+1)}\right|\leq\frac{C}{n^{1/p}}
\]
for all sufficiently large $n$.

Since $1< p<\infty$, Jensen's inequality implies
\[
\left(\frac{1}{n+1}\sum_{j=0}^n|\alpha_j|\right)^p\leq\frac{1}{n+1}\sum_{j=0}^n|\alpha_j|^p\leq\frac{\|\bsa\|_{\ell^p}^p}{n+1}
\]
and thus
\begin{equation}\label{1bound}
\sum_{j=0}^n|\alpha_j|\leq (n+1)^{1-1/p}\|\bsa\|_{\ell^p}.
\end{equation}
Similar reasoning shows that if $\bsa_N=\{\alpha_j\}_{j=N}^{\infty}$, then
\begin{equation}\label{Nbound}
\sum_{j=N}^n|\alpha_j|\leq (n-N+1)^{1-1/p}\|\bsa_N\|_{\ell^p}.
\end{equation}

Fix $\lambda\in\bbR$ and let $\tau=\lambda+\frac{C}{(n+1)^{1/p}}$ for some positive constant $C$.  If we use \eqref{sumrecur} to calculate $\eta_n(\tau)-\eta_n(\lambda)$, then we find
\begin{align*}
\eta_n(\tau)-\eta_n(\lambda)&=C(n+1)^{1-1/p}-2\sum_{j=0}^{n-1}\left(\arg\left[1-\alpha_jb_j(\eitau)\right]-\arg\left[1-\alpha_jb_j(\eilambda)\right]\right)
\end{align*}
Since $\tau\rightarrow\lambda$ as $\nri$ (uniformly in $\lambda\in\bbR$) 
we know that for any fixed $N$, the first $N$ terms in the sum are $o(1)$ as $\nri$.  Thus, we can rewrite the above as
\begin{align*}
\eta_n(\tau)-\eta_n(\lambda)&=C(n+1)^{1-1/p}-2\sum_{j=N}^{n-1}\left(\arg\left[1-\alpha_jb_j(\eitau)\right]-\arg\left[1-\alpha_jb_j(\eilambda)\right]\right)+o(1)
\end{align*}
as $\nri$.  Notice that for any $x\in\bbR$ it holds that $|\arg[1-\alpha_jb_j(e^{ix})]|\leq\arcsin|\alpha_j|\leq\frac{\pi}{2}|\alpha_j|$ so we can write
\begin{align*}
\eta_n(\tau)-\eta_n(\lambda)&\geq C(n+1)^{1-1/p}-2\pi\sum_{j=N}^{n-1}|\alpha_j|+o(1)
\end{align*}
as $\nri$.  By \eqref{Nbound}, we can then write
\begin{align}\label{last}
\eta_n(\tau)-\eta_n(\lambda)&\geq C(n+1)^{1-1/p}-2\pi(n-N)^{1-1/p}\|\bsa_N\|_{\ell^p}+o(1)
\end{align}
as $\nri$.

Now set $\lambda=\theta_j^{(n+1)}$ and choose $N$ so large that $2\pi\|\bsa_N\|_{\ell^p}<C$ (this can be done since $\|\bsa\|_{\ell^p}<\infty$).  Since $p>1$, equation (\ref{last}) implies $\eta_n(\tau)-\eta_n(\theta_j^{(n+1)})>2\pi$ when $n$ is large.  Thus the interval $(\theta_j^{(n+1)},\theta_j^{(n+1)}+\frac{C}{(n+1)^{1/p}})$ contains the argument of a zero of $\Phi_{n+1}^{(\beta_n)}(z)$.  Since this conclusion holds for all $j$ and $C>0$ was arbitrary, this yields the desired result.
\end{proof}


\noindent\textit{Remark.}  The method used to prove Theorem \ref{zerospace} also applies in the case $p=1$ and provides a new proof of the uniform clock spacing result proved in \cite[Theorem 4.1]{FineI}.

\medskip

 Theorem \ref{zerospace} relates zero spacings of paraorthogonal polynomials with decay properties of the Verblunsky coefficients.  From a spectral theoretic point of view, this is a meaningful result for any $p>1$.  From a measure theoretic point of view, this theorem is most useful in the case $1<p\leq2$.  Indeed, \cite[Theorem 2.10.1]{OPUC1} tells us that very little can be deduced about the measure corresponding to $\bsa$ from the hypothesis $\|\bsa\|_{\ell^p}<\infty$ for $p>2$. When $p\leq2$ one can make meaningful conclusions.  The condition $\|\bsa\|_{\ell^2}<\infty$ is equivalent to the underlying measure of orthogonality being in the Szeg\H{o} class, which has many equivalent formulations (see \cite[Section 2.7]{OPUC1}).  The case $1<p<2$ has more subtle implications for the measure of orthogonality, which are discussed in \cite{DK05} (see also \cite[Section 2.12]{OPUC1}).


Using a nearly identical argument as was used to prove Theorem \ref{zerospace}, one can prove our next result, which shows what happens if we replace the $\ell^p$ condition of Theorem \ref{zerospace} with a weak-type $\ell^p$ condition.  To state it, we need to define some notation.

If $\bsa=\{\alpha_n\}_{n=0}^{\infty}$ is a sequence of Verblunsky coefficients satisfying $\alpha_n\rightarrow0$ as $\nri$, then we can define its \textit{decreasing rearrangement} $\tilde{\bsa}=\{\tilde{\alpha}_n\}_{n=0}^{\infty}$ so that $\bsa=\tilde{\bsa}$ as sets and $|\tilde{\alpha}_0|\geq|\tilde{\alpha}_1|\geq|\tilde{\alpha}_2|\geq\cdots$ (this is possible since $\alpha_n\rightarrow0$ as $\nri$).  With this notation, we can now state our next result.

\begin{theorem}\label{zerospace2}
Suppose $\bsa=\{\alpha_n\}_{n=0}^{\infty}$ is a sequence of Verblunsky coefficients satisfying $\alpha_n\rightarrow0$ as $\nri$ and suppose $\tilde{\bsa}=\{\tilde{\alpha}_n\}_{n=0}^{\infty}$ is its decreasing rearrangement.  Assume
\[
\sum_{n=0}^N|\tilde{\alpha}_n|\leq f(N)
\]  
for some increasing function $f$ satisfying $\lim_{\nri}f(n)=\infty$.  Using the same notation as in Theorem \ref{zerospace}, it hols that
\[
\limsup_{\nri}\left[\frac{n}{f(n)}\cdot\sup_{1\leq j\leq n+1}\left|\theta_{j+1}^{(n+1)}-\theta_{j}^{(n+1)}\right|\right]\leq4.
\]
\end{theorem}

\noindent\textit{Remark.}  It follows from Theorem \ref{zerospace2} that the distance between zeros of $\Phi_{n+1}^{(\beta_n)}(z)$ for large $n$ will be no more than a constant multiple of $f(n)/n$.

\begin{proof}[Sketch of Proof]
Much of the proof is the same as the proof of Theorem \ref{zerospace}, so we only include some details here that are different.  As in that proof, we need to estimate the sum
\[
2\sum_{j=0}^{n-1}\left(\arg\left[1-\alpha_jb_j(\eitau)\right]-\arg\left[1-\alpha_jb_j(\eilambda)\right]\right)
\]
with $\tau$ and $\lambda$ as in the proof of Theorem \ref{zerospace}.  As before, we can say that the first $N$ terms are $o(1)$ as $\nri$, so it suffices to consider
\[
2\sum_{j=N}^{n-1}\left(\arg\left[1-\alpha_jb_j(\eitau)\right]-\arg\left[1-\alpha_jb_j(\eilambda)\right]\right)
\]
for any $N$ of our choosing.  Let us fix $\varepsilon>0$ and then choose $N$ so large that $\arcsin|\tilde{\alpha}_N|\leq(1+\varepsilon)|\tilde{\alpha}_N|$.  Then we have
\begin{align*}
&\left|2\sum_{j=N}^{n-1}\left(\arg\left[1-\alpha_jb_j(\eitau)\right]-\arg\left[1-\alpha_jb_j(\eilambda)\right]\right)\right|\leq4\sum_{j=N}^{n-1}\arcsin|\alpha_j|\\
&\qquad\qquad\qquad\leq4\sum_{j=N}^{n-1}\arcsin|\tilde{\alpha}_j|\leq4(1+\varepsilon)\sum_{j=N}^{n-1}|\tilde{\alpha}_j|\leq4(1+\varepsilon)f(n)
\end{align*}
Using the fact that $\varepsilon>0$ is arbitrary, the rest of the proof follows as in the proof of Theorem \ref{zerospace}.
\end{proof}

\smallskip

\noindent\textbf{Example 1.}  Suppose $\bsa$ is such that
\[
|\tilde{\alpha}_n|\leq\frac{A}{(n+2)^q}
\]
for some $q\in(0,1)$.  Then
\[
\sum_{n=0}^N|\tilde{\alpha}_n|\leq \frac{AN^{1-q}}{1-q}
\]  
and hence Theorem \ref{zerospace2} tells us that as $\nri$, the largest spacing between zeros of $\Phi_{n+1}^{(\beta_n)}(z)$ will be no larger than a constant times $n^{-q}$.  Setting $q=1/2$ we see that a weak-type $\ell^2$ condition is all that is required to obtain a $O(n^{-1/2})$ estimate on zero spacings.

\smallskip

\noindent\textbf{Example 2.}  Suppose $\bsa$ is such that
\[
|\tilde{\alpha}_n|\leq\frac{A}{n+2}.
\]
Then
\[
\sum_{n=0}^N|\tilde{\alpha}_n|\leq A\log (N+2)
\]  
and hence Theorem \ref{zerospace2} tells us that as $\nri$, the largest spacing between zeros of $\Phi_{n+1}^{(\beta_n)}(z)$ will be no larger than a constant times $\log n/n$.  Notice the similarity between this conclusion and the conclusion in the second part of \cite[Theorem 3]{Golquad}.

\smallskip

\noindent\textbf{Example 3.}  Suppose $\bsa$ is such that
\[
|\tilde{\alpha}_n|\leq\frac{A}{\log(n+2)}.
\]
Then one can calculate that there is a constant $B>0$ so that for large $N$ it holds that
\[
\sum_{n=0}^N|\tilde{\alpha}_n|\leq \frac{BN}{\log N}
\]  
Theorem \ref{zerospace2} tells us that as $\nri$, the largest spacing between zeros of $\Phi_{n+1}^{(\beta_n)}(z)$ will be no larger than a constant times $1/\log n$.

\section{A Generalized Mhaskar-Saff Theorem for POPUC}\label{mspopuc}

Notice that the proof of Theorem \ref{zerospace} shows that if $C>0$ and $p>1$, then for large $n$ the interval $[\lambda,\lambda+\frac{C}{(n+1)^{1/p}}]$ contains the arguments of many zeros of $\Phi_n^{(\beta)}(z)$ for any $\lambda\in\bbR$ and $\beta\in\partial\bbD$.  Our goal in this section is to make this statement more precise and provide conditions under which an arc must contain a positive fraction of the zeros of $\Phi_n^{(\beta)}(z)$.  In other words, if we define
\begin{equation}\label{nudef}
d\nu_n=\frac{1}{n}\sum_{j=1}^n\delta_{\theta_j^{(n)}},
\end{equation}
to be the counting measure of the arguments of the zeros of $\Phi_n^{(\beta_{n-1})}$, then we want to find intervals that must receive positive weight from any weak-$*$ limit of the measures $\{\nu_n\}_{n=1}^{\infty}$.

An example of such a result is known as the Mhaskar-Saff Theorem for POPUC\footnote{The Mhaskar-Saff Theorem applies to the zeros of OPUC and can be found in \cite[Theorem 2.3]{MhaSaff}.} and states that if
\[
\lim_{\jri}\frac{1}{n_j}\sum_{k=0}^{n_j-1}|\alpha_k|=0
\]
for some sequence $n_j$ increasing to infinity then the asymptotic distribution of the zeros of $\Phi_{n_j}^{(\beta_{n_j-1})}(z)$ is uniform on the unit circle for any sequence of $\beta_{n_j}\in\partial\bbD$ (see \cite[Theorem 8.1.2 $\&$ 8.2.7]{OPUC1}).  
We will prove the following generalization of that result.

\begin{theorem}\label{newmsaff}
Suppose $\{\alpha_n\}_{n\geq0}$ is a collection of Verblunsky coefficients satisfying
\begin{equation}\label{rlim}
\lim_{\jri}\frac{1}{n_j}\sum_{k=0}^{n_j-1}|\alpha_k|=r\in[0,1)
\end{equation}
for some sequence $\{n_j\}_{j=1}^{\infty}$ increasing to infinity.  Suppose $\nu$ is a weak-$*$ limit of the measures $\{\nu_{n_j}\}_{j=1}^{\infty}$ defined in \eqref{nudef}.  Then any interval of length $\Delta$ has $\nu$ measure in the interval $\left[\frac{\Delta}{2\pi}-r,\frac{\Delta}{2\pi}+r\right]$.  In particular, every interval of length larger that $2\pi r$ has positive $\nu$ measure.
\end{theorem}

\noindent\textit{Remark.}  If $\Delta<2\pi r$, then Theorem \ref{newmsaff} places no lower bound on the $\nu$-measure of the interval.  Similarly, if $\frac{\Delta}{2\pi}+r>1$, then Theorem \ref{newmsaff} places no upper bound on the $\nu$-measure of the interval.

\smallskip

\noindent\textit{Remark.}  Notice that by setting $r=0$ in Theorem \ref{newmsaff} we obtain a new proof of the Mhaskar-Saff Theorem for POPUC.

\begin{proof}
Fix $\Delta\in(0,2\pi)$, $\phi\in(-\pi,\pi]$, $\varepsilon>0$, and let
\[
\delta_k=\eta_k(\phi+\Delta+\varepsilon)-\eta_k(\phi).
\]
As in the proof of Theorem \ref{zerospace}, we have
\[
\delta_{n_j}=(n_j+1)(\Delta+\varepsilon)-2\sum_{k=0}^{n_j-1}\left(\arg\left[1-\alpha_kb_k(e^{i(\phi+\Delta+\varepsilon)}\right]-\arg\left[1-\alpha_kb_k(e^{i\phi})\right]\right).
\]
Therefore, we have
\begin{eqnarray*}
\delta_{n_j}\leq(n_j+1)(\Delta+\varepsilon)+4\sum_{k=0}^{n_j-1}\arcsin(|\alpha_k|)\leq(n_j+1)(\Delta+\varepsilon)+2\pi\sum_{k=0}^{n_j-1}|\alpha_k|.
\end{eqnarray*}
This implies
\[
\nu((\phi,\phi+\Delta))\leq\limsup_{\jri}\frac{\delta_{n_j}}{2\pi n_j}=\limsup_{\jri}\left[\frac{\Delta+\varepsilon}{2\pi}+\frac{1}{n_j}\sum_{k=0}^{n_j-1}|\alpha_k|\right]=\frac{\Delta+\varepsilon}{2\pi}+r.
\]
Since $\varepsilon>0$, $\Delta$, and $\phi$ were arbitrary, this proves the desired upper bound.

Applying similar reasoning with $\Delta+\varepsilon$ replaced by $\Delta-\varepsilon$ shows
\[
\delta_{n_j}\geq(n_j+1)(\Delta-\varepsilon)-2\pi\sum_{k=0}^{n_j-1}|\alpha_k|.
\]
This implies
\[
\nu((\phi,\phi+\Delta))\geq\liminf_{\jri}\frac{\delta_{n_j}}{2\pi n_j}=\liminf_{\jri}\left[\frac{\Delta-\varepsilon}{2\pi}-\frac{1}{n_j}\sum_{k=0}^{n_j-1}|\alpha_k|\right]=\frac{\Delta-\varepsilon}{2\pi}-r.
\]
Since $\varepsilon>0$, $\Delta$, and $\phi$ were arbitrary, this proves the desired lower bound.
\end{proof}

Our next example shows that the lower bound in Theorem \ref{newmsaff} is sharp.

\subsection*{Example: Geronimus Polynomials}  Consider the case when $\alpha_n\equiv\alpha\in(-1,0)$.  The corresponding orthogonal polynomials are known as the Geronimus polynomials and are discussed in detail in \cite[Section 1.6]{OPUC1}.  Clearly one can apply Theorem \ref{newmsaff} in this setting with $r=|\alpha|$.  In this case, one finds that the interval $\theta\in(-2\arcsin|\alpha|,2\arcsin|\alpha|)$ does not contain any zeros of $\Phi_n^{(-1)}(\eitheta)$.  This interval has length $4\arcsin r$ and if $r$ is very close to $1$, then this interval has length very close to $2\pi r$.  Thus, we can find intervals whose length is arbitrarily close to $2\pi r$ and which have $\nu$-measure $0$ for every weak-$*$ limit as in Theorem \ref{newmsaff}.  This shows that the lower bound in Theorem \ref{newmsaff} is sharp.

\subsection*{Application: Random Verblunsky Coefficients}  Consider the situation in which the Verblunsky coefficients are chosen randomly as in \cite{DS06,KiKo,KN,KStoi,Stoiciu,Stoiciu2,Stoiciu3}.  In this case, the zeros of the corresponding paraorthogonal polynomials form a collection of random points on the unit circle.  The randomness of these zeros depends heavily on the randomness of the Verblunsky coefficients.

We will be primarily interested in the case when the sequence $\{\alpha_n\}_{n=0}^{\infty}$ is a sequence of i.i.d. random variables.  In this case, it is known that for certain rotation invariant distributions, the zeros of the corresponding POPUC exhibit Poisson statistics (see \cite{Stoiciu,Stoiciu2,Stoiciu3}).  This indicates that the zeros exhibit a lack of correlation and behave as independent random points on the unit circle.  In this case, one expects to see clumps and gaps in the distribution of zeros of $\Phi_n^{(\beta)}(z)$ when $n$ is large (see \cite[Figure 1]{Stoiciu3} for a picture of this phenomenon).  We can apply Theorem \ref{newmsaff} to place upper bounds on how large a gap we might see in any random setting.

Indeed, if the Verblunsky coefficients are i.i.d., then the Strong Law of Large Numbers indicates that \eqref{rlim} holds almost surely with $n_j=n$ and for some $r\in(0,1)$ (we exclude the case $r=0$ because that case is trivial).  Thus, we expect that for large $n$, any arc of length at least $2\pi r$ will contain a zero (in fact many zeros) of the paraorthogonal polynomial $\Phi_n^{(\beta)}(z)$ for any $\beta\in\partial\bbD$.

\newpage





\end{document}